\documentclass[a4paper,twoside, reqno]{amsart}

\usepackage{amsmath}
\usepackage{amsthm}
\usepackage{amscd}
\usepackage{amssymb}
\usepackage{array}
\usepackage{hyperref}
\usepackage[pdftex]{graphicx}
\usepackage[latin1]{inputenc}
\usepackage{latexsym}
\usepackage{stmaryrd}
\usepackage[all]{xy}

\newtheorem{theorem}{Theorem}[section]
\newtheorem{lemma}[theorem]{Lemma}
\newtheorem{corollary}[theorem]{Corollary}
\newtheorem{definition}[theorem]{Definition}
\newtheorem{proposition}[theorem]{Proposition}
\newtheorem{remark}[theorem]{Remark}

\renewcommand{\dim}[0]{\operatorname{dim}}
\newcommand{\her}[0]{\operatorname{her}}
\newcommand{\image}[0]{\operatorname{im}}
\newcommand{\kernel}[0]{\operatorname{ker}}

\title{Subalgebras of finite codimension in semiprojective $C^*$-algebras}
\date{May 12, 2014}

\author{Dominic Enders}
\thanks{This work was supported by the SFB 878 {\itshape Groups, Geometry and Actions} and the Danish National Research Foundation through the {\itshape Centre for Symmetry and Deformation} (DNRF92).}

\address{Dominic Enders
\newline Department of Mathematical Sciences, University of Copenhagen
\newline Universitetsparken 5, DK-2100 Copenhagen \O, Denmark}
\email{d.enders@math.ku.dk}

\begin{document}

\begin{abstract}
We show that semiprojectivity of a $C^*$-algebra is preserved when passing to $C^*$-subalgebras of finite codimension. 
In particular, any pullback of two semiprojective $C^*$-algebras over a finite-dimensional $C^*$-algebra is again semiprojective.
\end{abstract}

\maketitle

\section{Introduction}

Since its introduction in the 1980's, the concept of semiprojectivity has become one of the most frequently used technical tools in the theory of $C^*$-algebras.
Originally, Blackadar defined semiprojective $C^*$-algebras as generalizations of ANR-spaces (absolute neighborhood retracts) in order to extend classical shape theory to a non-commutative setting (\cite{Bla85}).
Since then shape theory for $C^*$-algebras has been well studied and its connection to other homology theories, especially to $E$-theory, has been worked out by Dadarlat (\cite{Dad94}).

Nowadays, semiprojectivity is most often used for technical purposes since it gives the right framework to formulate and study perturbation questions for $C^*$-algebras.
This concept has found applications in numerous branches of $C^*$-theory including various classification programs. 
It is, for instance, essential in the classification of fields of $C^*$-algebras (\cite{Dad09}) or the classification of $C^*$-algebras coming from graphs (\cite{ERR13}).
It is further an important tool in the Elliott classification program as illustrated by the use of semiprojective building blocks in the construction of models for classifiable $C^*$-algebras. 

It is therefore desirable to have a sufficient supply of semiprojective $C^*$-algebras.
However, finding concrete examples or verifying semiprojectivity for a given object turns out to be surprisingly difficult.
One reason for this is the lack of closure properties for the class of semiprojective $C^*$-algebras. 
In fact, semiprojectivity is in general not preserved under most standard $C^*$-algebraic constructions.

In this paper we provide a new permanence result for semiprojectivity. 
It is shown that semiprojectivity passes to subalgebras of finite codimension (Corollary \ref{subalgebra}).
This includes the typical situation of a pullback of two semiprojective $C^*$-algebras over a finite-dimensional $C^*$-algebra (Corollary \ref{pullback}).

The proof of our result is based on parts of the work of Eilers, Loring and Pedersen in \cite{ELP98}. 
We combine their semiprojectivity results for certain 1-NCCWs (one-dimensional non-commutative CW-complexes) with results on universal extensions of finite-dimensional $C^*$-algebras.
This combination allows us to extend lifting problems from finite codimension ideals while keeping track of an intermediate subalgebra.
This directly implies that semiprojectivity passes to ideals and, together with an existence result by T. Katsura, even to subalgebras of finite codimension.

The author would like thank T. Katsura for his proof of Lemma \ref{Katsura} which improved the main result of this paper.

\section{Ideals and subalgebras of finite codimension}\label{sec preliminaries}
In this section we study a particular $C^*$-algebra containing a subalgebra of finite codimension.
It is shown in Proposition \ref{generic} that this special case can be implemented into any other $C^*$-algebra which contains a finite codimension subalgebra.
One should therefore think of it as the universal such situation.\\

First we fix some notation, following the one used in \cite{ELP98}. Given a unital $C^*$-algebra $F$ we write
\[S_1F=\{f\in\mathcal{C}_0([0,1),F)|f(0)\in\mathbb{C}1\},\]
\[C_1F=\{f\in\mathcal{C}([0,1],F)|f(0)\in\mathbb{C}1\}\]
and given a $C^*$-subalgebra $G\subseteq F$ (with a fixed embedding) we further set
\[C_1(F|G)=\{f\in\mathcal{C}([0,1],F)|f(0)\in\mathbb{C}1,f(1)\in G\}.\]

Now assume that $F$ is finite-dimensional. In this case, as shown by Eilers, Loring and Pedersen in \cite[Section 2]{ELP98}, the extension
\[0\rightarrow S_1F \rightarrow C_1F \rightarrow F \rightarrow 0\]
is in a sense the universal unital extension of $F$. 
More precisely, they showed how to implement the above sequence into any given unital extension of $F$ by the following Urysohn type result.

Recall that a $^*$-homomorphism $\alpha\colon A\rightarrow B$ is called {\itshape proper} if the image of an approximate unit $(u_\lambda)$ for $A$ is an approximate unit for $B$. 
Most important here is the fact that such maps extend via $\overline{\alpha}(m)b=\lim_\Lambda\alpha(mu_\lambda)b$ to $^*$-homomorphisms $\mathcal{M}(A)\rightarrow\mathcal{M}(B)$ between the corresponding multiplier algebras.
As shown in \cite{ELP99}, one can obtain functoriality properties for Busby maps with respect to proper homomorphisms. 
This eventually leads to the existence of pushout diagrams, i.e. amalgamated free products, as in the following case.

\begin{lemma}[2.3.3 in \cite{ELP98}]\label{Urysohn}
For each extension $A$ of $F$, where $A$ is unital and separable and $\dim(F)<\infty$, there is commutative diagram of extensions
\[\begin{xy}\xymatrix{
0 \ar[r] & I \ar[r] & A \ar[r] & F \ar[r] & 0 \\
0 \ar[r] & S_1F \ar[r] \ar[u]_\alpha & C_1F \ar[r] \ar[u]_{\overline{\alpha}} & F \ar[r] \ar@{=}[u] & 0
}\end{xy}\]
such that $\overline{\alpha}$ is a unital $*$-homomorphism whose restriction $\alpha$ to $S_1F$ is a proper $^*$-homomorphism to $I$. In particular, the left square is a pushout diagram.
\end{lemma}

We will need the following, slightly extended version of this result which in addition keeps track of a $C^*$-subalgebra of finite codimension.
It follows from this proposition together with Lemma \ref{Katsura} that $S_1F\subseteq C_1(F|G)\subseteq C_1F$, or more precisely
\[\begin{xy}\xymatrix{
0 \ar[r] & S_1F \ar[r] & C_1F \ar[r] & F \ar[r] & 0 \\
0 \ar[r] & S_1F \ar[r] \ar@{=}[u] & C_1(F|G) \ar[r] \ar[u]_\subseteq & G \ar[r] \ar[u]_\subseteq & 0
}\end{xy},\]
is the universal situation of a $C^*$-algebra containing a $C^*$-subalgebra of finite codimension.

\begin{proposition}\label{generic}
Suppose we are given a commutative diagram of extensions
\[\begin{xy}\xymatrix{
0 \ar[r] & I \ar[r] & A \ar[r] & F \ar[r] & 0 \\
0 \ar[r] & I \ar@{=}[u] \ar[r] & B \ar[u]_{\subseteq} \ar[r] & G \ar[u]_{\subseteq} \ar[r] & 0
}\end{xy}\]
with unital and separable $A$, a unital inclusion $B\subseteq A$ and $\dim(F)<\infty$. Then there exists a commutative diagram
\[\scalebox{0.9}{\begin{xy}\xymatrix{
& 0 \ar[rr] & & I \ar[rr] & & A \ar[rr] & & F \ar[rr] & & 0 \\
0 \ar[rr] & & S_1F \ar[rr] \ar[ur]^{\alpha} & & C_1F \ar[rr] \ar[ur]^{\overline{\alpha}} & & F \ar[rr] \ar@{=}[ur] & & 0 & \\
& 0 \ar[rr]|!{[ur];[dr]}\hole & & I \ar[rr]|!{[ur];[dr]}\hole \ar@{=}[uu]|!{[ul];[ur]}\hole & & B \ar[rr]|!{[ur];[dr]}\hole \ar[uu]|!{[ul];[ur]}\hole & & G \ar[rr] \ar[uu]|!{[ul];[ur]}\hole & & 0 \\
0 \ar[rr] & & S_1F \ar[rr] \ar[ur] \ar@{=}[uu] & & C_1(F|G) \ar[rr] \ar[ur] \ar[uu] & & G \ar[rr] \ar@{=}[ur] \ar[uu] & & 0 & 
}\end{xy}}\]
such that $\alpha$ is a proper $^*$-homomorphism. In particular, both the left square on the top and the left square on the bottom are pushout diagrams.
\end{proposition}

\begin{proof}
Lemma \ref{Urysohn} provides us with the exact row of the upper front of the diagram and with the $^*$-homomorphism $\overline{\alpha}\colon C_1F\rightarrow A$ which restricts to a proper homomorphism $\alpha\colon S_1F\rightarrow I$ and makes the top face of the diagram commute. 
One now verifies that $\overline{\alpha}$ maps $C_1(F|G)$ to $B$.
The statement on the pushout diagrams then follows from \cite[Corollary 4.3]{ELP99}.
\end{proof}

Assume $A$ is a $C^*$-algebra which contains a subalgebra $B$ of finite codimension.
In order to apply the proposition above, we need to know that there is an ideal $I$ of $A$ which is contained in $B$ and still has finite codimension.
The existence of such an ideal is obvious in most situations, for example if $B=C\oplus_F D\subseteq C\oplus D=A$ where the pullback is taken over a finite-dimensional $C^*$-algebra.
However, as T. Katsura observed, ideals like this always exist and we are indebted to him for the proof of this fact.

\begin{lemma}[Katsura]\label{Katsura}
Let $A$ be a $C^*$-algebra and $B$ a $C^*$-subalgebra of finite codimension in $A$. Then there exists an ideal $I$ of finite codimension in $A$ which is contained in $B$.  
\end{lemma}

\begin{proof}
Consider $I:=\{x\in B:xA\subseteq B\}$, then $IA,BI\subseteq B$ and $I$ is a (closed) right ideal in $A$.
We claim that $I$ is in fact a two-sided ideal, i.e., we also have $AI\subseteq B$.
Given $x\in I$ we also have $x^*x\in I$ and therefore $|x|^{\frac{1}{2}}=(x^*x)^{\frac{1}{4}}\in I$.
Now write $x=y|x|^{\frac{1}{2}}$ with $y\in B$ using the polar decomposition of $x$.
Then for any $a\in A$ we find $(ax)^*=|x|^{\frac{1}{2}}y^*a^*\in B$ and since $B$ is selfadjoint also $ax\in B$.

Now let $A$ act by left multiplication on the finite dimensional quotient vector space $A\slash B$. 
More precisely, we consider the linear map $\pi\colon B\rightarrow\mathcal{L}\left(A\slash B\right)$ given by $\pi(x)(a+B)=xa+B$. 
Then $\kernel(\pi)=\{x\in B:xA\subseteq B\}=I$ and $\dim(B/I)=\dim(\image(\pi))\leq\dim\left(\mathcal{L}\left(A\slash B\right)\right)=n^2<\infty$.  
\end{proof}

We finish this section with the following easy lemma which allows us to restrict ourselves to the case of essential ideals.
Recall that an ideal $I$ in a $C^*$-algebra $A$ is said to be {\itshape essential} if its annihilator $I^\bot$ in $A$ is trivial, i.e. if the canonical map $A\rightarrow\mathcal{M}(I)$ is injective.

\begin{lemma}\label{essential}
Suppose $I$ is an ideal of finite codimension in a $C^*$-algebra $A$. Then there exists a decomposition $A=A'\oplus G$ such that $I$ is essential in $A'$ and $G$ is a finite-dimensional $C^*$-subalgebra of $A$ orthogonal to $I$.
\end{lemma}

\begin{proof}
By assumption we have a short exact sequence
\[\begin{xy}\xymatrix{0 \ar[r] & I \ar[r] & A \ar[r]^p & F \ar[r] & 0}\end{xy}\]
with $F$ finite-dimensional. 
Define $G$ to be the annihilator of $I$ in $A$, i.e.
\[G:=I^{\bot}=\{a\in A:aI=Ia=\{0\}\}.\]
The quotient map $p$ is isometric on $G$ since, using an approximate unit $u_{\lambda}$ for $I$, we have $\|p(a)\|=\inf_{\lambda}\|a(1-u_{\lambda})\|=\|a\|$ for every $a\in G$.
Hence $G$ is finite-dimensional and in particular unital. 
Denote the unit of $G$ by $e$, then $A$ decomposes as $(1-e)A(1-e)\oplus G$ since $G$ is also an ideal in $A$. 
It is clear that $I$ is essential in $A':=(1-e)A(1-e)$. 
If we denote $p(A')$ by $F'$, we further get a decomposition of the quotient map
\[\begin{xy}\xymatrix{
0 \ar[r] & I \ar[r] & A \ar[r]^p & F \ar[r] & 0 \\ 
0 \ar[r] & I \ar@{=}[u] \ar[r] & A'\oplus G \ar@{=}[u] \ar[r]^{p'\oplus id} & F'\oplus G \ar@{=}[u] \ar[r] & 0
}\end{xy}.\]
\end{proof}

\section{A new permanence result for semiprojectivity}\label{section main}
Here we extend the surprisingly short list of permanence properties for the class of semiprojective $C^*$-algebras.
Before outlining the strategy for proving the new closure result, we remind the reader of the necessary definitions.
More detailed information on lifting properties for $C^*$-algebras can be found in Loring's book \cite{Lor97}.

\begin{definition}[{\cite[Definition 2.10]{Bla85}}]\label{def sp}
A separable $C^*$-algebra is semiprojective if for every $C^*$-algebra $B$ and every increasing chain of ideals $J_n$ in $B$ with $J_\infty=\overline{\bigcup_n J_n}$, and for every $^*$-homomorphism $\varphi\colon A\rightarrow B/J_\infty$ there exists $n\in\mathbb{N}$ and a $^*$-homomorphism $\overline{\varphi}\colon A\rightarrow B/J_n$ making the following diagram commute: 
\[\begin{xy}\xymatrix{& B \ar@{->>}[d]^{\pi_0^n} \\ & B/J_n \ar@{->>}[d]^{\pi_n^\infty} \\ A \ar[r]^\varphi \ar@{-->}[ur]^{\overline{\varphi}} & B/J_\infty}\end{xy}\] 
\end{definition}

Equivalently, one may define semiprojectivity as a lifting property for maps to certain direct limits (cf. \cite[Chapter 14]{Lor97}): 
An increasing sequence of ideals $J_n$ in $B$ gives an inductive system $(B/J_n)_n$ with surjective connecting maps $\pi_n^{n+1}\colon B/J_n\rightarrow B/J_{n+1}$ and limit isomorphic to $B/J_\infty$. 
On the other hand, it is easily seen that every such system gives an increasing chain of ideals $(\kernel(\pi_1^n))_n$. 
Hence, semiprojectivity of a $C^*$-algebra $A$ is equivalent to being able to lift maps $\varphi$ as in
\[\begin{xy}\xymatrix{
& D_n \ar@{->>}[d]^{\pi_n^\infty} \\
A \ar[r]^(0.4)\varphi \ar@{-->}[ur]^{\overline{\varphi}} & \varinjlim D_n
}\end{xy}\]
to a finite stage, provided that all connecting maps $\pi_n^m\colon D_n\rightarrow D_m$ of the system are surjective. 
In this paper we will work in this picture.\\

The idea for Theorem \ref{main} can be roughly outlined as follows:
Given a lifting problem $\varphi\colon B\rightarrow\varinjlim D_n$ for a subalgebra $B$ of a $C^*$-algebra $A$, one tries to extend this to a lifting problem $\overline{\varphi}\colon A\rightarrow\varinjlim E_n$ for the larger algebra.
Now if $A$ is known to be semiprojective one can solve this new lifting problem, i.e. find a lift $\theta\colon A\rightarrow E_n$ as indicated below.
\[\begin{xy}\xymatrix{
A \ar@{-->}[r]^(0.4){\overline{\varphi}} \ar@{..>}@/^1.5pc/[rr]^\theta & \varinjlim E_n & E_n \ar@{->>}[l]\\
B \ar@{}[u]|{\rotatebox{90}{$\subseteq$}} \ar[r]^(0.4)\varphi & \varinjlim D_n \ar@{}[u]|{\rotatebox{90}{$\subseteq$}} & D_n \ar@{->>}[l] \ar@{}[u]|{\rotatebox{90}{$\subseteq$}}
}\end{xy}\]
The hope is that the restriction of $\theta$ to $B$ takes values in $D_n$ and therefore solves the original lifting problem for $B$.
This is of course not always the case but it can be arranged in the case where $B=I$ is an ideal and the quotient $A/I$ enjoys sufficient lifting properties as well.
In this case one can choose $E_n$ to be the multiplier algebra of $D_n$ but has be extra careful about non-compatibility of the limit and multiplier construction involved. 
Keeping track of a subalgebra sitting between $I$ and $A$ is the tricky part in proving our main result.

\begin{theorem}\label{main}
Let $I$ be an ideal of finite codimension in a semiprojective $C^*$-algebra $A$. Then any subalgebra $B$ of $A$ which contains $I$ is also semiprojective.
\end{theorem}

\begin{proof}
Let $C^*$-algebras $I\subseteq B\subseteq A$ as in the statement of the theorem be given. 
We may assume that both $A$ and $B$ are unital and share the same unit. 
Using Lemma \ref{essential} we find compatible decompositions $B=B'\oplus H_B\subseteq A'\oplus H_A=A$ such that $H_B$ and $H_A$ are finite-dimensional and $I$ is an essential ideal of $A'$ (and hence also of $B'$).
Since $A$ (resp. $B$) is semiprojective if and only if $A'$ (resp. $B'$) is semiprojective, we may assume that $I$ is an essential ideal of $A$ (and $B$).

First we apply Lemma \ref{generic} to implement the generic case $S_1F\subseteq C_1(F|G)\subseteq C_1F$ described in section \ref{sec preliminaries} into our situation $I\subseteq B\subseteq A$.
Here we denote by $G=B/I\subseteq A/I=F$ the finite-dimensional quotients.
We get a commutative diagram with exact rows
\begin{equation}\label{hugediagram1} \scalebox{0.78}{\xymatrix{
& 0 \ar[rr] & & I \ar[rr]^i & & A \ar[rr]^p & & F \ar[rr] & & 0  &\\
0 \ar[rr] & & S_1F \ar[rr]^{\iota} \ar[ur]^{\alpha} & & C_1F \ar[rr]^{\pi} \ar[ur]^{\overline{\alpha}} & & F \ar[rr] \ar@{=}[ur] & & 0 & \\
& 0 \ar[rr]|!{[ur];[dr]}\hole & & I \ar[rr]|!{[ur];[dr]}\hole \ar@{=}[uu]|!{[ul];[ur]}\hole & & B \ar[rr]|!{[ur];[dr]}\hole \ar[uu]|!{[ul];[ur]}\hole & & G \ar[rr] \ar[uu]|!{[ul];[ur]}\hole & & 0 \\
0 \ar[rr] & & S_1F \ar[rr] \ar[ur] \ar@{=}[uu] & & C_1(F|G) \ar[rr] \ar[ur] \ar[uu] & & G \ar[rr] \ar@{=}[ur] \ar[uu] & & 0 & }
}\tag*{($\ast$)}\end{equation}
where all upward arrows are inclusions, $\alpha$ is a proper $^*$-homomorphism and the upper and lower left squares are pushouts.
We denote the inclusion maps for the two sequences on the top by $i$, resp. $\iota$, and the quotient maps by $p$, resp. $\pi$.

Now let an isomorphism $\varphi\colon B\rightarrow \varinjlim D_n$ be given, where the direct limit is taken over an inductive system of separable $C^*$-algebras $D_n$ with surjective connecting homomorphisms $\pi_n^m\colon D_n\rightarrow D_m$.
The induced (surjective) homomorphisms $D_n\rightarrow\varinjlim D_n$ will be denoted by $\pi_n^\infty$. 
We will construct a partial lift for $\varphi$ in order to prove semiprojectivity of $B$. 

The $C^*$-algebra $C_1(F|G)$ is known to be semiprojective by \cite[Theorem 6.2.2]{ELP98} (but see also Remark \ref{1-NCCW}), hence we can find a $^*$-homomorphism $\psi\colon C_1(F|G)\rightarrow D_{n_0}$ for some integer $n_0$ which makes the diagram
\[\begin{xy}\xymatrix{
& & D_{n_0} \ar@{->>}[d]^{\pi_{n_0}^\infty} \\
C_1(F|G) \ar@/^1pc/@{-->}[urr]^{\psi} \ar[r]^(0.6){\overline{\alpha}} & B \ar[r]^(0.4){\varphi} & \varinjlim D_n
}\end{xy}\]
commute.
Writing $\psi_n=\pi_{n_0}^n\circ\psi_{n_0}$ for $n\geq n_0$ we may now assume that
\[D_n=\her\left(\psi_n(S_1F)\right)+\psi_n(C_1(F|G))\]
since otherwise we just replace $D_n$ by the $C^*$-algebra on the right hand side of the equation.
In order to do so, note that the restriction of $\pi_n^m$ to these new algebras is still surjective and that we do not change the limit $\varinjlim D_n$ since properness of $\alpha$ implies that 
\[\begin{array}{rl}
&\varinjlim\left(\her\left(\psi_n(S_1F)\right)+\psi_n(C_1(F|G))\right)\\
=&\her\left((\pi_n^\infty\circ\psi_n)(S_1F)\right)+(\pi_n^\infty\circ\psi_n)(C_1(F|G)) \\ 
=&\her\left((\varphi\circ\alpha)(S_1F)\right)+(\varphi\circ\beta)(C_1(F|G))\\ 
=&\varphi(\her(\alpha(S_1F))+\beta(C_1(F|G)))\\
=&\varphi(I+\beta(C_1(F|G)))\\
=&\varphi(B)=\varinjlim D_n
\end{array}\]
where $\her(E)$ denotes the hereditary subalgebra $\overline{EDE}$ generated by a $C^*$-subalgebra $E$ of $D$.
This way we find ideals
\[D_n':=\her\left(\psi_n(S_1F)\right)\triangleleft D_n\]
which are easily seen to be essential ones. 
The restrictions of $\pi_n^m$ to $D_n'$ are still surjective with $\varinjlim D_n'=\varphi(\her(\alpha(S_1F))=\varphi(I)$.

Important for us is that the restriction of $\psi_n$ as a homomorphism from $S_1F$ to $D_n'$ is now proper and hence extends to a homomorphism $\overline{\psi}_n$ making left square on the top of the diagram
\begin{equation}\label{hugediagram2} \scalebox{0.67}{\xymatrix{
& 0 \ar[rr]& & D_n' \ar[rr]^{\iota_n'} & & \mathcal{M}(D_n') \ar[rr]^{\varrho_n'} & & \mathcal{Q}(D_n') \ar[rr] & & 0&\\
0 \ar[rr] & & S_1F \ar[rr] \ar[ur]^{\psi_n} & & C_1F \ar@{-->}[ur]^{\overline{\psi_n}} \ar[rr] & & F \ar[rr] \ar@{-->}[ur]^{\overline{\overline{\psi_n}}} & & 0\\
& 0 \ar[rr]|!{[ur];[dr]}\hole & & D_n' \ar[rr]|!{[ur];[dr]}\hole \ar@{=}[uu]|!{[ul];[ur]}\hole & & D_n \ar[uu]^(0.4){\iota_n}|!{[ul];[ur]}\hole \ar[rr]|!{[ur];[dr]}\hole & & D_n/D_n' \ar[rr] \ar@{..>}[uu]|!{[ul];[ur]}\hole & & 0\\
0 \ar[rr] & & S_1F \ar[rr] \ar[ur] \ar@{=}[uu] & & C_1(F|G) \ar[ur]^{\psi_n} \ar[uu] \ar[rr] & & G \ar[rr] \ar[uu] \ar@{..>}[ur] & & 0
}}\tag*{($\ast\ast$)}\end{equation}
commute.
Here $\iota'_n$ is the canonical inclusion of $D_n'$ into its multiplier algebra, $\varrho_n'$ the correponding quotient map and $\iota_n$ the canonical inclusion coming from the fact that $D_n'$ is an essential ideal in $D_n$.
One checks that for all $f\in C_1(F|G)$ and $g\in S_1F$ 
\[\overline{\psi}_n(f)\psi_n(g)=\psi_n(fg)=\psi_n(f)\psi_n(g)=(\iota_n\circ\psi_n)(f)\psi_n(g)\] 
holds, so that by properness of $\psi_n\colon S_1F\rightarrow D_n'$ the maps $\overline{\psi_n}$ and $\iota_n\circ\psi_n$ agree on the subalgebra $C_1(F|G)$, i.e. the middle square in the diagram above commutes as well.
Hence also the square of induced maps between the quotients, indicated by the dotted arrows on the right side of the diagram, commutes. 
The induced map on $F$ will be denoted by $\overline{\overline{\psi_n}}$. Note that $\overline{\overline{\psi_n}}$ is injective when restricted to $G$ because $D_n'$ is essential in $D_n$.

Let us now focus on the multiplier algebras $\mathcal{M}(D_n')$.
By \cite[Theorem 2.3.9]{WO93}, each $\pi_n^m\colon D_n'\rightarrow D_m'$ extends naturally to a surjective homomorphism $\overline{\pi}_n^m\colon\mathcal{M}(D_n')\rightarrow\mathcal{M}(D_m')$.
Hence $(\mathcal{M}(D_n'),\overline{\pi}_n^m)$ forms a new inductive system with surjective connecting maps.
Of course, this also gives an inductive structure $(\mathcal{Q}(D_n'),\overline{\overline{\pi}}_n^m)$ for the corresponding quotients, the corona algebras $\mathcal{Q}(D_n')$.
We further have an embedding $\iota_\infty'$ of $\varinjlim D_n'$ as an ideal in $\varinjlim\mathcal{M}(D_n')$ which is induced by the maps $\iota_n'$. 
Similarly, the maps $\iota_n$ induce an inclusion $\iota_\infty\colon\varinjlim D_n\rightarrow \varinjlim\mathcal{M}(D_n')$ since they are compatible with both limit structures. 
Next, we will show that the new inductive system of multiplier algebras provides a lifting problem for $A$.
Using the pushout-situation in the upper left square of \ref{hugediagram1}, the pair $(\iota_{\infty}'\circ\varphi,\overline{\pi}_
n^\infty\circ\overline{\psi}_n)$ defines a homomorphism $\overline{\varphi}$ as indicated below
\[\begin{xy}\xymatrix{
& & & D_n' \ar[dl]^{\pi_n^\infty} \ar[d]^{\iota_n'} \\
S_1F \ar[r]^>>>>>{\alpha} \ar[d]_{\iota} \ar@/^/[urrr]^(0.4){\psi_n} & I \ar[r]^<<<<<{\varphi} \ar[d]|(0.65)\hole_i & \varinjlim D_n' 	\ar[d]^{\iota_{\infty}'}|(0.3)\hole & \mathcal{M} (D_n') \ar[dl]^{\overline{\pi}_n^\infty} \\
C_1F \ar[r]^>>>>>{\overline{\alpha}} \ar@/^/[urrr]^(0.4){\overline{\psi_n}} & A \ar@{-->}[r]^<<<<<{\overline{\varphi}} & \varinjlim\mathcal{M}(D_m')& 
}\end{xy}\]
since both maps are compatible on $S_1F$, meaning $\iota_\infty'\circ\varphi\circ\alpha=\overline{\pi}_n^\infty\circ\overline{\psi}_n\circ\iota$, 
and by that give rise to $\overline{\varphi}$ satisfying $\overline{\varphi}\circ i=\iota_\infty'\circ\varphi$. 
Because of the pushout situation in the lower left square of \ref{hugediagram1} and
\[\begin{array}{rl}
\overline{\varphi}\circ\overline{\alpha}_{|C_1(F|G)}&=\overline{\pi}^\infty_n\circ(\overline{\psi}_n)_{|C_1(F|G)}\\
&=\overline{\pi}^\infty_n\circ\iota_n\circ\psi_n \\
&=\iota_\infty\circ\pi^\infty_n\circ\psi_n
\\&=\iota_\infty\circ\varphi\circ\overline{\alpha}_{|C_1(F|G)},
\end{array}\]
the restriction of $\overline{\varphi}$ to $B$ factors as $\iota_{\infty}\circ\varphi$.
We end up in the following situation
\[\scalebox{0.75}{\xymatrix{
& 0 \ar[rr] & & \varinjlim D_n' \ar[rr]^{\iota_{\infty}'} & & \varinjlim\mathcal{M}(D_n') \ar[rr]^{\varrho_{\infty}'} & & \varinjlim\mathcal{Q}(D_n') \ar[rr] & & 0 \\
0 \ar[rr] & & I \ar[rr] \ar[ur]^{\varphi} & & A \ar[rr] \ar[ur]^{\overline{\varphi}} & & F \ar[rr] \ar@{..>}[ur]^{\overline{\overline{\varphi}}} & & 0 & \\
& 0 \ar[rr]|!{[ur];[dr]}\hole & & \varinjlim D_n' \ar[rr]|!{[ur];[dr]}\hole \ar@{=}[uu]|!{[ul];[ur]}\hole & & \varinjlim D_n \ar[rr]|!{[ur];[dr]}\hole \ar[uu]^(0.4){\iota_{\infty}}|!{[ul];[ur]}\hole & & \varinjlim (D_n/D_n') \ar[rr] \ar[uu]|!{[ul];[ur]}\hole & & 0 \\
0 \ar[rr] & & I \ar[rr] \ar[ur] \ar@{=}[uu] & & B \ar[rr] \ar[ur]^{\varphi} \ar[uu] & & G \ar[rr] \ar@{..>}[ur] \ar[uu] & & 0 & 
}}\]
where all rows are exact, every square commutes and in each inductive system all connecting maps are surjective.
The dotted arrows indicate the maps induced by $\varphi$ and $\overline{\varphi}$. 
Let $\overline{\overline{\varphi}}$ denote the map coming from $\overline{\varphi}$, then by finite-dimensionality of $F$ we find $\kernel\left(\overline{\overline{\varphi}}\right)\subseteq\kernel\left(\overline{\overline{\psi}}_n\right)$ for large enough $n$. 
Hence we may assume that $\overline{\overline{\pi}}_n^\infty$ is injective on the image of $\overline{\overline{\psi}}_n$.

We now pass to a suitable subsystem of $(\mathcal{M}(D_n'))_n$. Consider the subalgebras
\[E_n:=\varrho_n^{-1}\left(\overline{\overline{\psi}}_n(F)\right)=\iota_n'(D_n')+\overline{\psi}_n(C_1F)\subseteq\mathcal{M}(D_n').\]
One easily checks that the restrictions of $\overline{\pi}_n^m$ to this subsystem are again surjective and that the limit $\varinjlim E_n=\overline{\pi}^\infty_n(E_n)$ contains $\overline{\varphi}(A)$. 
Using diagram \ref{hugediagram2} we further find
\[\iota_n(D_n)=\varrho_n^{-1}\left(\overline{\overline{\psi}}_n(G)\right)\subseteq E_n.\]

Finally, we are able to use our main assumption, semiprojectivity of $A$, to find a homomorphism $\theta$ that lifts $\overline{\varphi}$ to some $E_n$:
\[\begin{xy}\xymatrix{
& 0 \ar[r] & D_n \ar@{->>}[d]^{\pi_n^\infty} \ar[r]^{\iota_n'} & E_n \ar@{->>}[d]^{\overline{\pi}_n^\infty} \ar[r]^(0.4){\varrho_n'} & \overline{\overline{\psi}}_n(F) \ar@{->>}[d]^{\overline{\overline{\pi}}_n^\infty}_\sim \ar[r] & 0 \\
& 0 \ar[r] & \varinjlim D_n \ar[r]^{\iota'_\infty} & \varinjlim E_n \ar[r]^(0.4){\varrho_{\infty}'} & \overline{\overline{\varphi}}(F) \ar[r] & 0 \\
0 \ar[r] & I \ar[r]^i \ar[ur]^(0.4){\varphi} & A \ar[r]^p \ar[ur]^(0.4){\overline{\varphi}} \ar@{..>}@/^1pc/[uur]^(0.7)\theta|(0.37)\hole & F \ar[r] \ar[ur]^(0.4){\overline{\overline{\varphi}}} & 0 \\
}\end{xy}\]
The crucial point is that this lift will automatically map the subalgebra $B$ to $\iota_n(D_n)\subseteq E_n$. 
This follows from
\[(\overline{\overline{\pi}}_n^\infty\circ\varrho_n'\circ\theta)(B)=(\varrho_\infty'\circ\overline{\pi}_n^\infty\circ\theta)(B)=(\varrho_\infty'\circ\overline{\varphi})(B)=\overline{\overline{\varphi}}(G)\]
and the fact that $\overline{\overline{\pi}}_n^\infty$ is injective on $\overline{\overline{\psi}}_n(F)$. 
Therefore one finds $(\varrho_n'\circ\theta)(B)\subseteq\overline{\overline{\psi}}_n(G)$, in other words $\theta(B)\subseteq \iota_n(D_n)$.
By injectivity of $\iota_n$, we may now regard $\theta_{|B}$ as a map to $D_n$. 
One then immediately verifies 
\[\iota_{\infty}\circ\pi_n^\infty\circ\theta_{|B}=\overline{\pi}_n^\infty\circ\iota_n\circ\theta_{|B}=\overline{\varphi}=\iota_{\infty}\circ\varphi.\] 
By injectivity of $\iota_{\infty}$ this means $\pi_n^\infty\circ\theta=\varphi$, i.e. we have found a solution $\theta_{|B}$ to our original lifting problem $\varphi\colon B\rightarrow\varinjlim D_n$ and by that shown that $B$ is semiprojective.
\end{proof}

Combining Lemma \ref{Katsura} with Theorem \ref{main} we now obtain

\begin{corollary}\label{subalgebra}
A $C^*$-subalgebra of finite codimension in a semiprojective $C^*$-algebra is semiprojective.
\end{corollary}

The following is the most typical situation in which Corollary \ref{subalgebra} applies, we therefore state it explicitly: 
Assume we are given two semiprojective $C^*$-algebras $A$ and $B$ together with $*$-homomorphisms $\varphi\colon A\rightarrow F$ and $\psi\colon B\rightarrow F$ to a finite-dimensional $C^*$-algebra $F$.
Then the pullback $A\oplus_F B$ along $\varphi$ and $\psi$ is a subalgebra of finite codimension in the semiprojective $C^*$-algebra $A\oplus B$. Hence we have 

\begin{corollary}\label{pullback}
If $A$ and $B$ are semiprojective $C^*$-algebras, any pullback of $A$ and $B$ over any finite-dimensional $C^*$-algebra is also semiprojective. 
\end{corollary}

\begin{remark}\label{1-NCCW}
One of the most important examples of semiprojective $C^*$-algebras is the class of 1-NCCWs (one-dimensional non-commutative CW-complexes), defined in \cite{ELP98} as pullbacks of the form 
\[\begin{xy}\xymatrix{\text{1-NCCW} \ar@{-->}[r] \ar@{-->}[d] & G \ar[d] \\ \mathcal{C}([0,1],F) \ar[r]^\partial & F\oplus F}\end{xy}\]
with $F$ and $G$ finite-dimensional $C^*$-algebras and $\partial$ evaluation at the endpoints of the interval $[0,1]$.

Their original proof of semiprojectivity for 1-NCCW's (see \cite[sections $5-6$]{ELP98}) is rather instransparent while Corollary \ref{pullback} gives a very natural explanation for this fact.
Although we used their result in the proof Theorem \ref{main}, we actually only need to know semiprojectivity for 1-NCCWs of the special form $C_1(F|G)$.
This again can be deduced from a version of \ref{main} for finite codimension ideals together with rather elementary methods from \cite{Bla85} and \cite{LP98}.
In fact, the proof of \ref{main} simplifies a lot if one restricts to the case $I=B$ and it only requires semiprojectivity of the dimension-drop algebras $S_1F$ (which was already shown in \cite{Lor96}).
Therefore, Theorem \ref{main} can be used to give a simplified proof for semiprojectivity of 1-NCCWs.    
\end{remark}

\begin{remark}
The result of Theorem \ref{main} in the case of an ideal of codimension 1 is closely related to a conjecture by Blackadar. 
In \cite{Bla04} he conjectured that for an extension
\[\begin{xy}\xymatrix{
0 \ar[r] & I \ar[r] & A \ar[r] & \mathbb{C} \ar[r] & 0
}\end{xy}\]
semiprojectivity of $I$ implies semiprojectivity of $A$.
While Eilers and Katsura (\cite{EK}, see also \cite{Sor12}) were able to construct a counterexample to this conjecture, Theorem \ref{main} shows that the converse implication holds in general.
\end{remark}

\begin{remark}\label{Busby}
The strategy of extending lifting problems as in \ref{main} can be used to obtain more general permanence results for semiprojectivity.
In fact, given an extension $0\rightarrow I\rightarrow A\rightarrow A/I\rightarrow 0$ one can show that semiprojectivity of $A$ and $A/I$ implies semiprojectivity for $I$ provided that in addition the Busby map $\tau\colon A/I\rightarrow\mathcal{Q}(I)$ associated to the extension has good lifting properties.
It is implicitly used in \ref{main} that this is the case whenever $\tau$ has finite-dimensional image.
A general study of lifting properties for Busby maps will be discussed elsewhere.
\end{remark}

\end{document}